\documentclass[preprint,12pt]{elsarticle}

\usepackage{amsmath,amsfonts,amsthm}
\usepackage{amssymb,latexsym}
\usepackage[colorlinks]{hyperref}
\usepackage{graphicx}
\usepackage[all]{xy}
\usepackage{layout}

\newtheorem{thm}{Theorem}[section]
\newtheorem{cor}[thm]{Corollary}
\newtheorem{lem}[thm]{Lemma}
\newtheorem{prop}[thm]{Proposition}
\theoremstyle{definition}
\newtheorem{defn}[thm]{Definition}

\theoremstyle{remark}

\numberwithin{equation}{section}
\include{xypic}

\DeclareMathOperator{\im}{Im}

\DeclareMathOperator{\id}{Id}
\newcommand{\A}{\mathcal{A}}
\newcommand{\B}{\mathcal{B}}
\newcommand{\C}{\mathcal{C}}
\newcommand{\W}{\mathcal{W}}
\newcommand{\X}{\mathcal{X}}
\newcommand{\Y}{\mathcal{Y}}

\newcommand{\p}{\mathcal{P}}

\newcommand{\To}{\longrightarrow}

\newcommand{\G}{\mathcal{G}}
\newcommand{\Go}{\mathcal{G}^{(0)} }
\newcommand{\gad}{(\G,\{A_u\}_{u\in \G^{(0)}}, \alpha)}
\newcommand{\gbd}{(\G,\{B_u\}_{u\in \G^{(0)}}, \beta)}
\newcommand{\gccd}{(\G,\{C_u\}_{u\in \G^{(0)}}, \gamma)}
\newcommand{\gwd}{(\G,\{W_u\}_{u\in \G^{(0)}}, \omega)}
\newcommand{\phiu}{\{\varphi_u\}_{u\in \G^{(0)}}}
\newcommand{\psiu}{\{\psi_u\}_{u\in \G^{(0)}}}
\newcommand{\pu}{\{p_u\}_{u\in \G^{(0)}}}
\newcommand{\qu}{\{q_u\}_{u\in \G^{(0)}}}
\newcommand{\theu}{\{\theta_u\}_{u\in \G^{(0)}}}

\journal{Documenta Mathematica}

\begin{document}
\begin{frontmatter}

\title{$C^*$-simplicity and Ozawa conjecture for groupoid $C^*$-algebras, part I: injective envelopes}

\author{Massoud Amini, Farid Behrouzi}

\address{Department of Mathematics, Faculty of Mathematical Sciences, Tarbiat Modares University,
Tehran 14115-134, Iran\\
School of Mathematics, Institute for Research in Fundamental Sciences (IPM)\\
Tehran 19395-5746, Iran\\
mamini@modares.ac.ir, mamini@ipm.ir}

\address{Faculty of Mathematics, Alzahra University, Vanak,
Tehran 19938-91176, Iran\\
f\_behrouzi@alzahra.ac.ir}


\begin{abstract}
This paper studies injective envelopes of groupoid dynamical systems and the corresponding boundaries.  Analogue  to the group  case, we associate a bundle of compact Hausdorff spaces to any (discrete) groupoid (the Hamana boundary of the groupoid). We study bundles of compact topological spaces equipped with an action of a groupoid.  We show that  any  groupoid  has a minimal boundary (the Furstenberg boundary of the groupoid). We prove that the Hamana  and  Furstenberg boundaries are the same, for (discrete) groupoids. We find the relation between the reduced crossed product of the $\G$-injective envelop of a groupoid dynamical system and the injective envelope of the reduced crossed product of the original system.

\end{abstract}

\begin{keyword}
groupoid dynamical system\sep  injective envelop\sep  boundary


\MSC[2008] 46l35\sep 20F65

\vspace{.3cm}
The first author was partly supported by grants from IPM (94430215)

\end{keyword}

\end{frontmatter}

\tableofcontents



\section{Introduction}
Let $G$ be a discrete group and $V$ be an operator system, i.e., a unital self-adjoint closed subspsace of a unital C*-algebra. We say that $V$ is a $G$-module ($G$-operator system) if there is  a homomorphism $\alpha$ from $G$ into the group of all unital complete order isomorphisms of $V$. In this case, we say that the triple $(G, V, \alpha)$  is a (group) dynamical system. Hamana in \cite{ham1} studied injectivity in the category whose objects are $G$-operator systems and whose morphisms are completely positive unital $G$-homomorphisms. In this setting, he proved that every $G$-operator system $V$ has  a unique $G$-injective envelope $I_G(V)$, i.e., a minimal $G$-injective $G$-operator system containing $V$ as sub-$G$-operator system.  Hamana obtained the $G$-injective envelope of $V$ by first embedding $V$ into a $G$-injective operator system $W$. He then obtained a minimal $V$-projection of $W$ and proved that the $G$-injective envelope is the rang of this projection. For   $\mathbb{C}$  with the trivial action $G$, the $G$-injective envelope of $\mathbb{C}$ is $C(X)$ for a compact Hausdorff space $X$ \cite{hp}.  Moreover,  $X$ is a $G$-space, called the Hamana boundary of $G$, denoted by $\partial _H G$.   It is proved in \cite{kk} that the action of $G$ on $\partial _H G$ is a boundary action, i.e.,   it is minimal and strongly proximal (an action is minimal if it has dense orbits, and strongly proximal if for any probability measure $\mu$ on $X$, the weak*-closure of the orbit $G\cdot\mu$ contains a point mass; see  \cite{fur, g} for more details).

On the other hand, Furstenberg in \cite{fur} proved that for any discrete group $G$, there is  a unique (up to $G$ -isomorphism) maximal $G$-boundary  $\partial _F G$. Maximality here means that  every  $G$-boundary is a quotient of $\partial _FG$. This is called the Furstenberg bounadry of $G$. Kalantar and Kennedy in \cite{kk} proved that for a discrete group $G$, the Furstenberg and  Hamana boundaries are $G$-isomorphic. They also relate this to the notion of exactness of groups, introduced by Kirchberg and Wasserman \cite{kw}. Ozawa proved in
\cite[Theorem 3]{oz} that a discrete group is exact if and only it acts amenably on its Stone-$\check{C}$ech
compactification, and authors in \cite{kk} show the same for the action on the Furstenberg boundary. More generally, Wasserman \cite{w} showed that a $C^*$-algebra is exact if it can be embedded
into a nuclear $C^*$-algebra (the converse is also true by a result of Kirchberg \cite{k}). Ozawa conjectured in \cite{oz} that for an  exact $C^*$-algebra $A$, there is a nuclear $C^*$-algebra between $A$ and its injective envelope. One of the main objectives of \cite{kk} was to prove this for reduced $C^*$-algebras of discrete
exact groups. Along the way, they also contributed to the $C^*$-simplicity problem by showing that the reduced crossed product of $C(\partial _FG)$ by the canonical action of $G$  is simple if and only if the action is topologically free.

This paper seeks an appropriate extension of these notions and results to (discrete) groupoids. As far as we know, none  of the above results is explored for groupoids. The motivation of the paper is two folds. First we want to introduce appropriate notions of $\G$-boundary for groupoids and show that the Hamana and Furstenberg boundaries are the same (Theorem \ref{main}). Also, we would like to make tools for checking the $C^*$-simplicity and Ozawa conjecture for groupoid $C^*$-algebras (and crossed products), something which is pursued in a forthcoming paper \cite{ab}, in which the nuclearity of crossed products under exact groupoid actions is studied.

The paper is organized as follows. In Section 2, we introduce basic notions of  groupoids and groupoid dynamical systems. For an $r$-discrete groupoid $\G$, we describe  notions  such as $\G$-essentiality, $\G$-rigidity, $\G$-injectivity  and prove the existence and uniqueness of the injective envelope of  groupoid dynamical systems.
In Section 3, we proved that for any groupoid dynamical system $\A$, there is a minimal injective dynamical system $I_{\G}(\A)$ "containing" $\A$. We prove this by showing that every groupoid dynamical system $\A$ can be embedded
 into an injective dynamical system, and then   find a minimal $\A$-projection with $I_{\G}(\A) $ as its range. As an important example, the injective envelope of the trivial groupoid dynamical system with one dimensional fibers, gives a bundle $\{X_u\}_{u\in \Go}$ of compact Hausdorff spaces, called  the Hamana boundary of $\G$.
In section 4,  we consider an $r$-discrete groupoid $\G $  acting on a bundle  of compact Hausdorff spaces over the unit space $\Go$ of $\G$, and study minimality, strong proximality and boundary actions in this case. We show that there is a unique maximal $\G$-boundary, called the Furstenberg boundary.
We show that the Hamana boundary is $\G$-isomorphic to the Furstenberg boundary. We also study the (reduced) crossed products of the groupoid dynamical systems and show that the reduced crossed product of the $\G$-injective envelop of such a system is included in the injective envelope of the reduced crossed product. This is essential in \cite{ab}, where we want to prove the Ozawa conjecture for groupoid crossed products.

\section{Groupoid  dynamical systems}

We review basic facts on groupoids. For more details we refer the reader to \cite{muh, pat, renu}.
\begin{defn}\label{d1}
A groupoid is a set $\G$ endowed with a product map: $\G^2\To \G$; $(g, h)\mapsto gh$, where $\G^2$ is a subset of $\G\times \G$, called the set of composable pairs, and an inverse map: $\G\To \G$; $g\mapsto g^{-1}$  such that
\begin{enumerate}
\item $(g^{-1})^{-1}=g$,
\item if $(g, h)\in \G^2$ and $(h, k)\in \G^2$, then $(gh, k), (g, hk) \in \G^2$ and $(gh)k=g(hk)$,
\item $(g^{-1}, g)\in \G^2$ and if $(g, h)\in \G^2$, then $g^{-1}(gh)=h$,
\item $(g, g^{-1})\in \G^2$ and if $(h, g)\in \G^2$, then $(hg)g^{-1}=h$,
\end{enumerate}
for each $f,g,h\in\G$.
\end{defn}
The unit space $\G^0$ is the subset of elements $gg^{-1}$, where $g$ ranges over $\G$. The range and source maps $r:\G\To \G^0$ and $d: \G\To \G^0$ are defined by $r(g)=gg^{-1}$ and $d(g)=g^{-1}g$. A pair $(g, h)$ belongs to  $\G^2$  if and only if $d(g)=r(h)$. For each $u\in \G^0$, the subsets $\G_u$ and $\G^u$ are defined  by $\G_u=d^{-1}(\{u\})$ and $\G^{u}=r^{-1}(\{ u\})$.

An operator system is a closed, self-adjoint subspace of a unital $C^*$-algebra containing its unit, or equivalently, a closed, self-adjoint subspace of $B(H)$ containing the identity operator on the Hilbert space $H$. In the latter case, we say that $A$ is an operator system on $H$ (see \cite{er} for more details).
\begin{defn}
A groupoid dynamical system is a triple\[ \mathcal{A}=\gad,\] such that
\begin{enumerate}
\item $\G$ is a groupoid,
\item for each $u\in \Go$, $A_u$ is an operator system,
\item  for each $g\in \G$, $\alpha_g: A_{d(g)}\To A_{r(g)}$  is a complete order  isomorphism,
\item for each $ (g,h)\in \G^{(2)}$,  $\alpha_g\alpha_h=\alpha_{gh}$,
\end{enumerate}
for $g\in \G$ and $a\in A_{d(g)}$, we write  $g\cdot a$ for $\alpha_g(a)$.
\end{defn}


The following definition uses the notion of completely positive (c.p.) maps between $C^*$-algebras. The notion is also meaningful for maps between operator systems (see \cite{er} for more details). We assume that all completely positive maps are unital.

\begin{defn}
 A $\G$-morphism between  systems $\mathcal{A}=\gad$  and $\mathcal{B}=\gbd$ is a family $\{\varphi_u\}_{u\in \Go}$ of maps such that
 \begin{enumerate}
 \item for any  $u$,
 $\varphi_u: A_u\To B_u$ is a  c.p. map,
 \item  for any  $g\in \G$ and $a\in \G^{d(g)},$
 \[\beta_g(\varphi_{d(g)}(a))=\varphi_{r(g)}(\alpha_g(a)),\]
\end{enumerate}
i.e., for each $g$, the following diagram is commutative:
\[\xymatrix{A_{d(g)}\ar[r]^{\varphi_{d(g)}}\ar[d]_{\alpha_{g} } & B_{d(g)}\ar[d]^{\beta_{g}}\\
A_{r(g)}\ar[r] ^{\varphi_{r(g)}}& B_{r(g)}}.\]

The composition of $\G$-morphisms $\{\varphi_u\}_{u\in \Go}$  and $ \{\psi_u\}_{u\in \Go}$ (if it makes sense)  is defined by
 \[\{\psi_u\}_{u\in \Go}\circ \{\varphi_u\}_{u\in \Go}=\{\psi_u\circ \varphi_u\}_{u\in \Go}.\]
\end{defn}
Let
$\mathcal{A}=\gad, \ \ \ \mathcal{B}=\gbd, \ \ \ \mathcal{C}=\gccd, $ and $\mathcal{W}=\gwd$ be dynamical systems and
$\Phi=\phiu, \ \Psi=\psiu,$ and $\Theta=\theu$ be $\G$-morphisms. A $\G$-morphism $\Phi: \mathcal{A} \To\mathcal{B} $ is a $\G$-injection  (resp., a $\G$-isomorphism) if for any $u$, $\varphi_u : A_u\To B_u$ is an  injection  (resp., a  complete order  isomorphism).

\begin{defn}
A $\G$-extension of a  groupoid dynamical system $\A$ is a pair $(\W, \Theta)$ such that $\Theta: \A\To \W$  is a $\G$-embedding.
Moreover,
\begin{enumerate}
\item  $(\W, \Theta) $  is $\G$-essential  if for any $\G$-morphism $\Phi: \W\To \C$, $\Phi $ is a  $\G$-injection  whenever $\Phi\circ \Theta$ is a
$\G$-injection,
 \item $(\W, \Theta) $ is $\G$-rigid if for any $\G$-morphism $\Phi: \W\To \W$, $\Phi\circ \Theta=\Theta$ implies  $\varphi_u=id_{W_u}$, for all $u\in \Go$.
\end{enumerate}
\end{defn}
\begin{defn}
A groupoid dynamical system $\W$ is called $\G$-injective if for systems $\A$ and $\B$ and  $\G$-injective morphism $\Phi: \A\To \B$ and arbitrary $\G$-morphism $\Psi:\A\To\W$, there exists a $\G$-morphism
$\Theta:\B\To\W$ such that $\Theta\circ \Phi=\Psi$, that is, for any $u\in \Go$, the following diagram is commutative:
 \[\xymatrix{
A_u\ar[r]^{\varphi_u}\ar[d]^{\psi_u}&B_u\ar@{.>}[ld]^{\theta_u}\\
W_u.}\]
\end{defn}

\section{Injective envelopes and Hamana boundary}
In this section we explore the existence and uniqueness of injective objects in the category of groupoid dynamical systems. An operator system is a closed, self-adjoint subspace of a unital $C^*$-algebra containing its unit, or equivalently, a closed, self-adjoint subspace of $B(H)$ containing the identity operator on the Hilbert space $H$. In the latter case, we say that $A$ is an operator system on $H$ (see \cite{er} for more details). The injective envelopes of operator systems are studied by Hamana \cite{ham2}.

Let $X, X_0$ be sets and $s: X\To X_0$ be a surjective function. Let $\mathcal{A}=\{A_u\}_{u\in X_0}$ be a family of operator spaces. A section of $\mathcal A$ is a function $f: X\To \bigcup_{x\in X_0}$ such that $f(x)\in A_{s(x)}$ and
\[ \|f\|_{\infty}=\sup_{x\in X}\|f(x)\|< \infty.\]We denote the set of all sections of $\mathcal{A}$ by $\Gamma_{\infty}(X, s^*\A)$. The space $\Gamma_{\infty}(X, s^*\A)$ of sections form an is an operator system with pointwise operations and involution
and the above norm. We  denote the set of  functions with finite support in $\Gamma_{\infty}(X, s^*\mathcal{A}) $   by $\Gamma_c(X, s^*\mathcal{A}) $. This  is  a *-subspace of $\Gamma_{\infty}(X, s^* \mathcal{A})$. If each $A_u$  is a $C^*$-algebra, then $\Gamma_{\infty}(X, s^*\mathcal{A}) $ is a $C^*$-algebra and $\Gamma_c(X, s^*\mathcal{A}) $  is  a *-subalgebra.

\begin{lem}\label{ll1}
If $\mathcal{A}=\{A_u\}_{u\in X_0}$ is a family of injective operator
systems and $s: X\To X_0$ is  surjective, then $\Gamma_{\infty}(X, s^*\A)$ is an injective operator system.
\end{lem}
\begin{proof}
Let  $B$ and $C$ be operator spaces and $\theta : B\To C$ be an injective completely  positive map. To each completely  positive map $\varphi: B\To \Gamma_{\infty}(X, s^*\A)$  and $x\in X$ we associate the map $\varphi_x : B\To A_{s(x)}$ by $\varphi_x(b)=\varphi(b)(x)$. By injectivity of $A_{s(x)}$, there exists a completely positive map $\psi_x: C\To A_{s(x)}$ such that $\psi_x\circ \theta=\varphi_x$. Define $\psi: C\To \Gamma_{\infty}(X, s^*\A)$ by
$\psi(c)(x)=\psi_x(c)$, for $c\in C$ and $x\in X$. Then
\[\psi\circ\theta(b)(x)=\psi_x(\theta(b))=\varphi_x(b)=\varphi(b)(x),\]
for $b\in B$ and $x\in X$.
\end{proof}

Suppose that $\G$ is a groupoid and $\A=\{A_u\}_{u\in \Go} $ is a family of operator systems.
Then $(\G, \{\Gamma_{\infty}( \G^u, s^*\mathcal{A})\}_{u\in\Go}, \ell)$ becomes a groupoid dynamical system with the action
\[\ell_g: \Gamma_{\infty}( \G^{d(g)}, s^*\mathcal{A})\To\Gamma_{\infty}( \G^{r(g)}, s^*\mathcal{A}),\]
\[\ell_g(f)(x)=f(g^{-1}x).\]
For  $u\in \Go$, define
\[\im_u: A_u \To \Gamma_{\infty}( \G^u, s^*\mathcal{A})\]
by $\im_u (a)(g)=g^{-1}. a$. Then $\{\im_u\}_{u\in \Go}$ is an injective $\G$-morphism.

The next lemma extends \cite[Lemma 2.2]{ham1}.
\begin{lem}\label{ll2 }
If $\A=\{A_u\}_{u\in \Go} $ is a family of  injective operator systems, then \[(\G, \{ \Gamma_{\infty}(\G^u, s^*\A)\}_{u\in\Go}, \ell)\] is a $\G$-injective groupoid dynamical system.
\end{lem}
\begin{proof}
Let $\Theta :\B\To\C$ be a  $\G$-injective morphism and let
 \[\varphi: \B\To(\G, \{ \Gamma_{\infty}(\G^u, s^*\A)\}_{u\in \G^{(0)}}, \ell) \]
 be a $\G$ -morphism. For any $u\in \Go$, define $\widehat{\varphi}_u : B_u\To \Gamma_{\infty}(\G^u, s^*\A)$ by $\widehat{\varphi}_u(b)=\varphi_u(a)(u)$. By the injectivity of $A$ ,   there exists a completely positive map $\widehat{\psi}_u : C_u \To A_u$ such that $\widehat{\psi}_u\circ \theta_u=\widehat{\varphi}_u$. Define a completely positive  map $\psi_u: C_u\To  \Gamma_{\infty}(\G^u, s^*\A)\}$ by $\psi_u(b)(g)=\widehat{\psi}_{d(g)}(g^{-1}.b)$.  For any $u\in \Go$  ,  $g\in \G^{u}$ and $b\in B^u$, we have
\begin{align*}
\psi_u\circ \theta_u(b)(g)&=\widehat{\psi}_{d(g)}(g^{-1}.\theta_u(b))\\
&=\widehat{\psi}_{d(g)}(\theta_{d(g)}(g^{-1}.b))=\widehat{\varphi}_{d(g)}(g^{-1}.b)\\
&=\varphi_{d(g)}(g^{-1}.b)(d(g))=(g^{-1}.\varphi_u)(d(g))\\
&=\varphi_u(b)(gd(g))=\varphi_u(b)(g).
\end{align*}
Thus $\{\psi_u\}_{u\in \G^{(0)}}\circ \theu= \phiu$.  We show that $\{\psi_u\}_{u\in \Go}$ is a $\G$-morphism:
\begin{align*}
g.\psi_{d(g)}(c)(h)&=\psi_{d(g)}(c)(g^{-1}h)\\
&=\widehat{\psi}_{d(h)}(h^{-1}.g.c)=\psi_{r(g)}(g.c)(h),
\end{align*}
for $g\in \G$ ,  $h\in \G^{r(g)}$ and $c\in C_{d(g)}$. Thus $g.\widehat{\varphi}_{d(g)}(c)=\widehat{\varphi}_{r(g)}(g.c)$.
\end{proof}
The next result follows from the above two lemmas and some routine algebraic manipulations.
\begin{prop}
Let $\A$ be groupoid dynamical system. Then $\A$ is  injective if and only if for each $u\in \Go$,  $A_u$ is an injective operator system and there exists a  $\G$-morphism $$\phiu: (\G, \{ \Gamma_{\infty}(\G^u, s^*\A)\}_{u\in\Go}, \ell)\To \mathcal{A}$$ such that  $\varphi_u\circ\im_u=\id_{A_u}$, for any $u\in \Go$.
\end{prop}

We say that $\A$ is a $\G$-dynamical subsystem of $\B$, if for any $u\in \Go$, $A^u\subseteq B^u$, and for any $g\in \G$,
$\beta_g|_{A_{d(g)}}=\alpha_g$.
\begin{defn}
Let $\A$ be a $\G$ dynamical subsystem of $\B$.
\begin{enumerate}
\item  A $\G$-morphism $\Phi: \A\To \B$ is called an $\A$-projection if for any $u$, $\phi_u\circ \phi_u=\phi_u$ and $\phi_u|_{A_u}=id_{A_u}$,
\item A family $\pu$ of seminorms  is called an $\A$-seminorm if there exists a $\G$-morphism $\phiu$ such that,
for any $u$, $p_u( . )=\|\phi_u( .)\|$ and $\phi_u|_{A_u}=id_{A_u}$.
\end{enumerate}
\end{defn}
\begin{defn}
Let $\mathbf{P}$ (resp. $\mathbf{Pr}$)  be the set of  all $\A$-seminorms (resp., all $\A$-projections) on $\B$. We define  partial orders
  on $\mathbf{P}$ and $\mathbf{Pr}$ as follows:
$\pu \leq \qu $,  if for any $u\in \Go$  and $b\in B_u$, $p_u(b)\leq q_u(b)$, and
$\phiu\preceq \psiu$,  if for any $u\in \Go$, $ \psi_u\circ\phi_u=\phi_u\circ\psi_u=\phi_u$.
\end{defn}

The next lemma extends \cite[Lemma 3.4]{ham2}.

\begin{lem} \label{l1}
Let $\A$ be a $\G$-dynamical subsystem of an injective groupoid dynamical system $\B$. Then there exists a minimal $\A$-seminorm  on $\B$.
\end{lem}
\begin{proof}
Since $\{id_{B^u}\}_{u\in \Go}$ induces an $\A$-seminorm on  $\B$, by Zorn lemma, it is enough to show that every decreasing net of $\A$-seminorms on $\B$ has a lower bound.
Suppose that $\{\{p_{i, u}\}_{u\in \Go}\}_{i\in I}$ is such a decreasing net. For any $i\in I$, there exists a $\G$-morphism $\{\varphi_{i, u}\}_{u\in \Go}$ such that,  for any $u\in \Go$, $p_{i, u}(. )=\|\varphi_{i, u}( .)\|$ and $\varphi_{i, u}|_{A^u}=id_{A^u}.$ Put $B=\oplus_{u\in \Go} B_u$ and let $H$ be a Hilbert space such that $B\subseteq B(H)$.
Define $J: B\To \ell^{\infty}(\G, B(H))$ by $J((b_u)_{u\in \Go})(g)=(\tilde{b}_u)_{u\in \Go}$, where
\[\tilde{b}_u=\left\{
\begin{array}{lcr}
b_u&& u\neq d(g)\\
&&\\
g^{-1}\cdot b_{r(g)}&& u=d(g).
\end{array}
\right.\]
Then $J$ is an imbedding, and we may regard $B$ as an operator subsystem  of $\ell^{\infty}(\G, B(H))$. The restriction of $J$ to $B^u$ is the imbedding $J_u: B^u\To \ell^{\infty}(\G^u, B(H));\ J_u(b)(g)=g^{-1}\cdot b.$   For any $i\in I$,  define $\varphi_i: B\To B\subseteq \ell^{\infty}(\G, B(H))$ by
\[\varphi_i(b_u)_{u\in \Go}=(\varphi_{i, u}(b_u))_{u\in \Go}.\]
Then $\{\varphi_i\}_{i\in I}$  is a net in the unit ball of $B(B, \ell^{\infty}(\G, B(H)))$, which is compact in the point-weak$^*$ topology, thus there exists a subnet $\{\varphi_j\}_{j\in I'}$, point-weak$^*$-converging to some $\varphi_0$ in $B(B, \ell^{\infty}(\G, B(H)))$, that is,
for any $v\in \Go$ and $b\in B^v$,
\[\varphi_{j, v}(b)\To \varphi_0((\hat{b}_u(b))_{u\in \Go})(v),\]
where
\[\hat{b}_u(b)=\left\{
\begin{array}{lcr}
b&&u=v\\
&&\\
0&&u\neq v.
\end{array}
\right.\]
By the  injectivity of $\B$,  there  exists a $\G$-morphism  \[\Psi:( \G, \ell^{\infty}(G^u, B(H)), \ell)\To\B\] such that for any $u\in \Go$, $\psi_u\circ J_u=id_{B_u}$. Define \[\varphi_{0, v}:B_v\To \ell^{\infty}(G^v, B(H))\] by $\varphi_{0, v}(b)=\varphi_0((\hat{b}_u(b))_{u\in \Go})|_{\G^v}$ and
\[\varphi_{v}=\psi_v\circ \varphi_{0, v}.\]
Then $\varphi_v$ is a c.p. map from $B_v$ into $B_v$ and $\varphi_v|_{B_v}=id_{B_v}$. Let us observe  that $\Phi$ is a $\G$-morphism. Since $\Psi$ is a $\G$-morphism, it is enough to show that $\{\varphi_{0, u}\}_{u\in \Go}$ is a family of $\G$-morphisms.
To see this, suppose  that  $b\in  B_{d(g)}$, for some $g\in \G$. If $h\in G^{r(g)}$, then
\begin{align*}
g\cdot \varphi_{0, d(g)}&(b)(h)=\varphi_0((\hat{b}_u(b))_{u\in \Go})(g^{-1}h)\\
&=\lim_j \varphi_j ((\hat{b}_u(b))_{u\in \Go})(g^{-1}h)
=\lim_j  \varphi_{j, d(g)}(b)(g^{-1}h)\\
&=\lim_j h^{-1}\cdot g\cdot  \varphi_{j, d(g)}(b)
=\lim_j h^{-1}\cdot  \varphi_{j, r(g)}(g\cdot b)\\
&=h^{-1}\cdot \varphi_{0}((\hat{b}_u(g.b))_{u\in \Go})(r(g))
=\varphi_{0}((\hat{b}_u(g\cdot b))_{u\in \Go})(hr(g))\\
&=\varphi_{0}((\hat{b}_u(g\cdot b))_{u\in \Go})(h)=\varphi_{0, r(g)}(g\cdot b)(h).
\end{align*}
 Put $p_u(\cdot)=\| \varphi_u(\cdot)\|$, tnen $\{p_u\}_{u\in \Go}$ is a lower bound for $\{\{p_{i, u}\}_{u\in \Go}\}_{i\in I}$. For $b\in B^v$,
\begin{align*}
p_v(b)=\|\varphi_v(b)\|&=\|\psi_v\circ\varphi_{0, v}(b)\|\leq \|\varphi_{0, v}(b)\|\\
&=\|\varphi_{0}((\hat{b}_u(b))_{u\in \Go})\|\leq \limsup_j \|\varphi_{j}((\hat{b}_u(b))_{u\in \Go})\|\\
&=\limsup_j \| \varphi_{j, v}(b)\|
=\lim_{i}p_{i, v}(b).
\end{align*}
\end{proof}

Now we are able to extend \cite[Lemma 3.5]{ham2}.
\begin{thm} \label{t1}
Let $\A$ be a  $\G$-dynamical subsystem of $\B$ and $\B$ is $\G$-injective. Then there is a minimal $\A$-projection on $\B$.
\end{thm}
\begin{proof}
By Lemma \ref{l1}, there exists a minimal $\A$-seminorm $\{\tilde{p}_u\}_{u\in \Go}$. Thus,  for any $u\in \Go$,  the exists $\widetilde{\varphi}_u: B_u\To B_u$ such that $\widetilde{\varphi}_u|_{A_u}=id_{A_u}$ and $\tilde{p}_u(. )=\|\widetilde{\varphi}_u(. )\|$. Let
\[\varphi^{(n)}_u=\frac{1}{n}(\widetilde{\varphi}_u+\widetilde{\varphi_u}^2+\cdots+\widetilde{\varphi_u}^n).\]
Then $\{\{ \varphi_u^{(n)}\}_{u\in \Go}\}_{n\in \mathbb{N}}$ is a net of $\G$-morphisms from $\B$ into itself. A similar argument as in the proof of Lemma \ref{l1} shows that there exist  a subnet $\{\{ \varphi_u^{(n_j)}\}_{u\in \Go}\}_{j\in \mathbb{N}}$  and a $\G$-morphism
$\Phi$ such that $\varphi_u^{(n_j)}(b)\To \varphi_u(b)$, for all $u\in \Go$  and $b\in B_u$,    in the weak$^*$-topology. Take a  $\G$-morphism $\Psi$  which is an idempotent from $\B$ into $(\G, \ell^{\infty}(\G^u, B(H)),\ell)$, where $H$ is a Hilbert space with  $\oplus_u B_u\subseteq B(H)$.  For $u\in\Go$,
\begin{align*}
\|\psi_u\circ\varphi_u(b)\|&\leq \|\varphi_u(b)\|\leq\limsup_j \|\varphi_u^{(n_j)}(b)\|\leq \|\widetilde{\varphi}_u(b)\|=\tilde{p}_u(b).
\end{align*}
By the minimality of $\{\tilde{p}_u\}_{u\in \Go}$,  $\|\psi_u\circ\varphi_u(b)\|=\tilde{p}_u(b)$, thus
\[\limsup_j \|\varphi_u^{(n_j)}(b)\|
= \|\widetilde{\varphi}_u(b)\|.\]
Therefore,
\begin{align*}
\|\widetilde{\varphi}_u(x)-\widetilde{\varphi}_u^2(x)\|
&=\|\widetilde{\varphi}_u(x-\widetilde{\varphi}_u)\|\\
&=\limsup\|\varphi^{(n_j)}(x-\widetilde{\varphi}_u(x)\|\\
&=\limsup\frac{1}{n}\|\widetilde{\varphi}
(x)-\widetilde{\varphi}_{u}^{~n_j+1}(x)\|=0.
\end{align*}
Hence $\widetilde{\Phi}=\{\widetilde{\varphi}_u\}_{u\in \Go}$ is an $\A$-projection.
To see the minimality of $\widetilde{\Phi}$,  suppose that $\Theta$ is any $\A$-projection with
$\Theta\preceq \widetilde{\Phi}$. Then, for $u\in \Go$,
$\theta_u\circ \widetilde{\varphi}_u=\widetilde{\varphi}_u\circ\theta_u=\theta_u$. Thus
$\|\theta_u(b)\|\leq \|\widetilde{\varphi}_u(b)\|=\tilde{p}_u(b)$. The minimality of
$\{p_u\}_{u\in \Go}$ implies that, for any $u$, $\|\theta_u(b)\|=\|\widetilde{\varphi}_u(b)\|$, in particular,
$\ker \theta_u=\ker \widetilde{\varphi}_u$. For $b\in B_u$,
\[\widetilde{\varphi}_u(b)=\widetilde{\varphi}_u((b-\theta_u(b))+\theta_u(b))=\widetilde{\varphi}(
\theta_u(b))=\theta_u(b).\]
\end{proof}

The two next lemmas are proved  similar to Lemma \ref{t1} and   \cite[Lemma 3.6]{ham2}.
\begin{lem}\label{l3}
Let $\A$ be a groupoid $\G$-dynamical  subsystem of $\B$ and $\Phi: \B\To \B$  be a $\G$-morphism which induces  a minimal $\A$-seminorm. Then the extension
\[\mathcal{IM}(\Phi)=(( \G , \{\varphi_u(B_u)\}_{u\in \Go}, \beta ),\{i_u\}_{u\in \Go}), \] is $\G$-rigid, where $i_u: A_u\To \varphi_u(B_u)$ is inclusion map.
\end{lem}
\begin{lem}\label{l4}
Let $(\B, \Phi)$ be a $\G$-injective $\G$-extension of $\A$. Then $(\B, \Phi)$  is $\G$-rigid if and only if it is $\G$-essential.
\end{lem}

\begin{lem}\label{l5}
Let $\A$ be an injective groupoid dynamical system and $\Phi$ be an idempotent $\G$-morphism of $\A$. Then $(\G, \{\varphi_u(A^u)\}_{u\in \Go}, \alpha)$ is injective.
\end{lem}
\begin{proof}
Let $\B$ and $\C$ be groupoid dynamical systems,  $\Psi: \B\To\C$ be a $\G$-injective morphism and $\Theta:\B\To(\G, \{\varphi_u(A^u)\}_{u\in \Go}, \alpha)$ be a $\G$-morphism. Suppose that $\{i_u\}_{u\in \Go}:(\G, \{\varphi_u(A^u)\}_{u\in \Go}, \alpha)\To\A$ is  the inclusion morphism. Since $\A$ is injective, there exists a $\G$-morphism $\{\hat{\psi}_u\}_{u\in\Go}$ from $\C$ into $\A$ such that for any $u$,
\[\hat{\psi}_u\circ\psi_u=i_u\circ \theta_u.\]
Hence,
$\{\varphi_u\circ \hat{\psi}_u\}_{u\in \Go}$ is a $\G$-morphism from $\gccd $ into $\A$ such that, for any $u$,
\[(\varphi_u\circ \hat{\psi}_u)\circ \psi_u=\varphi_u\circ (\hat{\psi}_u\circ \psi_u) =\varphi_u\circ(i_u\circ\theta_u)=\theta_u.\]
\end{proof}

We are ready to prove the main result of this section.

\begin{thm}\label{t1}
Any groupoid dynamical system $\A$ has a $\G$-injective envelope $(I_{\G}(\A), \Upsilon)$, which is unique up to $\G$-isomorphism.
\end{thm}
\begin{proof}
Let $H$ be a Hilbert space with $\oplus_{u\in\Go} A^u\subseteq B(H)$, and put
\[\W=(\G, \{\ell^{\infty}(\G^u, B(H))\}_{u\in \Go}, \ell), \{\j_u\}_{u\in \Go}).\]
  For  $u\in \Go$, define   $J_u: A^u\To \ell^{\infty}(\G^u, B(H))$ by $J_u(a)(g)=g^{-1}\cdot a$. We may  regard $\A$ as a $\G$-dynamical  subsystem of  $\W$. By Theorem \ref{t1}, there exists a minimal $\A$-projection $\Theta$ on $\W$, and  $I_{\G}(\A)=(\G, \{\theta_u(\ell^{\infty}(\G^u, B(H)))\}_{u\in \Go}, \ell)$ is injective by Lemma \ref{l5}. Suppose that
  \[i_u: \A^u\To \theta_u(\ell^{\infty}(\G^u, B(H)))\]
  is  the inclusion map and $\Upsilon=\{i_u\}_{u\in \Go}$.
Then $(I_G(\A), \Upsilon)$ is a $\G$-injective envelope of $\A$, by Lemma \ref{l3}.
Now if  $(\B, \Phi)$ is any other  $\G$-injective envelop of $\A$, then there exist $\G$-morphisms $\Psi$ from $I_{\G}(\A)$ into
$\B$ and $\{\hat{\psi}_u\}_{u\in \Go}$  from $\B$ into  $I_{\G}(\A)$  such that, for  $u\in \Go$,
$\psi_u\circ i_u=\varphi_u$ and $\hat{\psi}_u\circ \varphi_u=i_u$, hence
$\psi_u\circ\psi_u\circ \varphi_u=\varphi_u$ and $\hat{\varphi}\circ\psi_u\circ i_u=i_u$. By the rigidity,
$\hat{\psi}_u\circ\psi_u=id_{\theta(\ell^{\infty}((\G^u, B(H))))}$ and $\psi_u\circ\hat{\psi}_u=id_{B_u}$.
\end{proof}
Our next step is to find an analog for the Hamana boundary. We first need the following result.

\begin{prop}
Let $\gad$ be a $\G$-injective groupoid dynamical system. Then for any $u\in \Go$,   there
is a unique multiplication
$\cdot : A_u\times A_u\To A_u$
making $A_u$ a $C^*$-algebra in its given $*$-operation and norm, and for any $g\in \G$,
$\alpha_g :A_{d(g)}\To A_{r(g)}$  is an isomorphism of $C^*$-algebras. Moreover, if for any $u\in \Go$,  $A_u$  is an operator system in a commutative $C^*$-algebra, then under this multiplication, each $A_u$ becomes a commutative $C^*$-algebra.
\end{prop}
\begin{proof}
Similar to the proof of  Lemma \ref{l1},   there is a Hilbert space $H$ such that
\[\W=(\G, \{\ell^{\infty}(\G^u, B(H))\}_{u\in \Go}, \ell)\]
has $\gad$ as a $\G$-dynamical subsystem. By injectivity, there is a $\G$-morphism $\phi_u: \W\To \A$ such that, for any $u\in \Go$,
$\phi_u$ is c.p. and $\phi_u|_{A_u}=id_{A_u}$. Given
$x,y\in \A_u=\phi_u(\ell^{\infty}(\G^u, B(H)))$, put $x\circ y= \phi_u(xy).$
By \cite[Theorem 6.1.3.]{er}, this operation defines a multiplication on $A_u$, making $A_u$  a C*-algebra. For  $g\in \G$ and $x, y\in A_{d(g)}$,
\[g\cdot (x\circ y)=g\cdot \phi_{d(g)}(xy)=\phi_{r(g)}(g\cdot (xy))=\phi_{r(g)}((g\cdot x)(g\cdot y))=
(g\cdot x)\circ (g\cdot y).\]
Moreover, if for any $u\in \Go$, there is a compact  Hausdorff space $X_u$ such that $A_u\subseteq C(X_u)$, put  $X=\bigsqcup_u X_u$,  then we may  regard $C(X_u)$ as a $C^*$-subalgebra of $\ell^{\infty}(X)$, and $\gad$ is a subsystem of
$(\G, \{\ell^{\infty}(\G^u, \ell^{\infty}(X))\}, \ell )$. If we define the multiplication on $A_u$ as above, then $A_u$ is a commutative $C^*$-algebra.
\end{proof}

\begin{cor}Let $\gbd$ be the injective envelope of a groupoid dynamical system $\gad$  such that, for any $u\in \Go$, $A_u$ is an operator system in  a commutative $C^*$-algebra. Then, for any $u\in \Go$, $B_u$ is commutative $C^*$-algebra.
\end{cor}

The bundle with one dimensional fibres is of special interest. Let $\G$ be a groupoid and for  $u\in\Go$, set $C_u=\mathbb{C}$. Then $\gccd$ is a groupoid dynamical system where $\gamma $ is the trivial action, that is, $\gamma_g :C_{d(g)}\To C_{r(g)}$ is the identity. Then the injective envelope of $\C$ is of the form $(\G, C(\partial_H^u), \beta)$, where $\partial_H^u$ is a compact Hausdorff space. We call $\{\partial_H^u\}_{u\in \Go}$ the \textit{Hamana boundary} of $\G$ and denote it by $\mathbf{\partial}_H (\G)$. For any $g\in \G$, $\gamma_g$ induces a homeomorphism $\gamma_g^*: \partial_H^{r(g)}\To \partial_H^{d(g)}$ such that, for any $f\in C(\partial_H^{d(g)})$,
$\gamma_g(f)(x)=f(\gamma^*_{g^{-1}}\cdot x).$
The groupoid dynamical system $\C$ is $\G$-injective if and only if there are states $\phi_u: \ell^{\infty}(\G^u)\To \mathbb{C}$ such that for any $g\in \G$ and $f\in \ell^{\infty}(G^{d(g)})$,
\ $\phi_{r(g)}(g\cdot f)=\phi_{d(g)}(f).$

\section{Furstenberg boundary}

The notion of groupoid action on  sets (or topological spaces) generalizes the concept of
group action by considering partially defined maps.
\begin{defn}
Let $\G$ be a groupoid. A $\G$-space is a bundle of  locally compact Hausdorff spaces $\mathcal{X}=\{X_u\}_{u\in \Go}$ and a bundle of maps $\{\alpha_g\}_{g\in \G}$ such that
\begin{enumerate}
\item for any $g\in \G$, $\alpha_g$ is a homeomorphism from  $X_{d(g)}$  onto $X_{r(g)}$,
\item for any $u\in \Go$, $\alpha_u$ is the identity map $id_{X_u}$,
\item for any $(g, h)\in \G^{(2)}$, $\alpha_g\circ\alpha_h=\alpha_{gh}$.
\end{enumerate}

We denote $\alpha_g(x)$ by $g\cdot x$. A $\G$-subspace of $\mathcal{X}$ is a bundle of locally compact Hausdorff  spaces $\mathcal{Y}=\{Y_u\}_{u\in \Go}$ such that, for each $u$,  $Y_u\subseteq X_u$ and, for each $g$, the restriction of $\alpha_g$ to $Y_{d(g)} $ is a homeomorphism onto $Y_{r(g)}$.
\end{defn}

For the rest of this section, all topological spaces are assumed to be compact and Hausdorff.

\begin{defn}
A $\G$-map between $\G$-spaces $\X$ and $\Y$ is  a family of maps $\{\phi_g\}_{g\in \G}$ such that
\begin{enumerate}
\item for any $g\in \G$,  $ \phi_g: X_{d(g)}\To Y_{d(g)}$  is continuous,
\item for any $g\in \G$ and $x\in X_{d(g)}$, $ g\cdot \phi_{d(g)}=\phi_{r(g)}(g\cdot x)$.
\end{enumerate}
\end{defn}
We denote  the space of complex finite Radon measures on $X$ by $M(X)$ and the subset of probability measures by $P(X)$, equipped with the weak$^*$-topology. There is natural embedding of $X$ into $P(X)$ as point masses. If $\mathcal{X}=\{X_u\}_{u\in \Go}$ is a $\G$-space, then $\mathcal{P}(\mathcal{X})=\{P(X_u)\}_{u\in \Go}$ is a $\G$-space. For $g\in G$ and $\mu\in P(X_{d(g)})$ define $g\cdot \mu(E)=\mu(g^{-1}\cdot E)$, for Borel subsets $E$  of $X_{r(g)}$.

Let $\X$ be a $\G$-space. Then for any $g\in \G$, the map $x\mapsto g\cdot x$  is a homeomorphism from $X_{d(g)}$ onto $X_{r(g)}$. This induces an $*$-isomorphism
\[\alpha_g : C(X_{d(g)})\To C(X_{r(g)});\ \ \alpha_g(f)(x)=f(g^{-1}\cdot x),\]
and $(\G, \{C(X_u)\} , \alpha)$ is a groupoid dynamical system. Conversely, given the groupoid dynamical system $(\G, \{C(X_u)\},  \alpha )$,  for $g\in \G$, $\alpha_g: C(X_{d(g)})\To C(X_{r(g)})$ is an *-isomorphism, and by Banach-Stone theorem, there exists a homeomorphism $\tilde{\alpha}_g: X_{r(g)}\To X_{d(g)}$ such that, for $f\in C(X_{d(g)})$,  $\alpha_g(f)(x)=f (\widetilde{\alpha}_g(x))$. For
\[\alpha_g^*: X_{d(g)}\To X_{r(g)};\ \ \alpha_g^*(x)=\tilde{\alpha}_{g^{-1}}(x),\]
$\{X_u\}_{u\in \Go}$ is a $\G$-space.

Let $\X$ and $\Y$ be $\G$-spaces. There is a one-to-one correspondence between $\G$-morphisms $\Phi :(\G, \{C(X_u)\}_{u\in\Go}, \alpha)\To (\G, \{C(Y_u)\}_{u\in\Go}, \beta)$ and $\G$-maps $\Phi^*:  \{Y_u\}_{u\in\Go}\To \{P(X_u)\}_{u\in \Go}$, given by
$\phi_g(f)(y)=\phi_g^*(y)(f).$ Here, the restriction of the adjoint map $\phi_g^*$ to $Y_{d(g)}$  is a continuous map from $Y_{d(g)} $ into $P(X_{d(g)})$.

\begin{defn}
A $\G$-space $\X$ is called minimal if  there is no  nontrivial $\G$-subspace, and strongly proximal if for every $u,v \in \Go$,  with $G_u^v\neq \emptyset$, and  $\mu\in P(X_u)$, $\overline{\G_u^v.\mu}\cap X_{v}\neq \emptyset$. A compact $\G$-space $X$ is called a $\G$-boundary   if it is minimal and strongly proximal, or equivalently, if
$ \X$ is the unique minimal  closed $\G$-subspace  of $\mathcal{P}(\X)$.
\end{defn}
By Zorn lemma,  every $\G$-space has a minimal $\G$-subspace. Also, every $\G$-subspace of a strongly  proximal $\G$-space is again strongly  proximal.

Let $\{\X^i=\{X^i_u\}_{u\in \Go}\}_{i\in I}$ be a family of $\G$-spaces. The product space
$\prod _{i\in I}\X^i=\{~\prod_{i\in I} X^i_u  ~\}_{u\in\Go}$ is a $\G$-space with the diagonal $\G$-action.
\begin{lem}\label{lemm2}
If $\{\mathcal{X}_i\}_{i\in I} $ is a family of compact strongly proximal $\G$-spaces, then $\prod_{i\in I}\mathcal{X}_i$ is also strongly proximal.
\end{lem}
\begin{proof}
For the case where  $I$ is finite, it suffices to prove check the claim when $I$ has two elements.
Let $\X=\{X_u\}_{u\in\Go}$  and $\Y=\{Y_u\}_{u\in\Go}$ be two strongly proximal $\G$-spaces. Let us show that $\X\times \Y=\{X_u\times Y_u\}_{u\in\Go}$  is strongly proximal.  Define $\Lambda _u: P(X_u\times Y_u)\To P(X_u)$ by
$\Lambda_u(\mu)(E)=\mu(E\times Y_u)$. Take $u,v\in \Go$ with $\G_u^v\neq \emptyset$ and $\mu\in P(X_u\times Y_u)$. It is easy to see that $\Lambda_u(\overline{\G_u^v\cdot\mu})=\overline{\G_u^v\cdot\Lambda_u(\mu)}.$  Since $\X$ is strongly proximal, there exists $x\in X_v$ such that $\delta_x\in \Lambda_u(\overline{\G_u^v\cdot\mu})$.  An straightforward measure theory argument shows that there exists $\nu\in P(X_v)$  such that $\delta_x\times \nu\in \overline{\G_u^v\cdot\mu}$.  Since $\Y$ is strongly proximal, there exists a net $\{g_i\}$ in $\G^v_v$ and $y\in X_v$ such that $g_i\cdot\nu\To y.$
By compactness, we may assume that there is $x'\in X_v$ such that $g_i\cdot x\To x'$. Therefore, $\delta_{x'
}\otimes \delta_y\in \overline{\G_v^v\cdot \overline{\G_u^v\cdot\Lambda_u(\mu)}}\subseteq \overline{\G_u^v\cdot\Lambda_u(\mu)}$.

For the general case, we need to use the idea of functions depending on finitely many variables. More precisely, let  $\{ X_i\}_{i\in I}$ be a family of compact Hausdorff spaces. For any finite subset $F\subseteq I$, let $C_F$ be the set of all continuous functions in $C(\prod_{i\in I} X_i)$  that depend only on  variables indexed by $F$, i.e., $f\in C_F$  if and only if  $f((x_i)_{i\in I})=f((y_i)_{i\in I})$,  whenever $(x_i)_{i\in I}, (y_i)_{i\in I}\in \prod_{i\in I} X_i$  with  $x_i=y_i$, for all $i\in F$. By Stone-Weierstrass theorem, $\bigcup_F C_F$ is dense in $C(\prod_{i\in I} X_i)$, where the union is taken over all finite subsets of $I$. Therefore, if  $P\subseteq P(\prod_{i\in I} X_i) $, then
\[\overline{P}^{w^*}=\bigcap_F \overline{P}^F,\]
where $\overline{P}^F$ is the closure of $P$ in the weak topology on $P(\prod_{i\in I} X_i) $   induced by $C_F$.
Let $F\subseteq I$ be  a finite subset and let  $(x_i)_{i\in I\setminus F}\in \prod_{i\in I\setminus F} X_i $. If $\mu \in P(\prod_{i\in I} X_i)$, then there exists $\mu_F\in \prod_{i\in  F} X_i $ such that, for all $f\in C_F$,
\[\int _{\prod_{i\in I} X_i}f~d\mu= \int_{{\prod_{i\in I} X_i}} f ~d(\mu_F\times \delta_{(x_i)_{i\in I\setminus F}}).\]

Next, if $I$ is an arbitrary set and $u,v\in \Go$ such that $\G_u^v\neq \emptyset,$ and $\mu\in P(\prod_{i\in I} X^i_u)$, by the above observation, for any finite subset $F\subseteq  I$, $\overline{G_u^v\cdot\mu}^F\bigcap \prod_{i\in I} X^i_v\neq \emptyset $ and hence $\overline{G_u^v\cdot\mu}^{w^*}\bigcap\prod_{i\in I} X^i_v\neq \emptyset$, by the Cantor intersection theorem.
\end{proof}

Now we could extend \cite[Proposition 4.2]{fur}.

\begin{lem}\label{lemm3}
Let $\X$ be a $\G$-boundary and $\Y$ be a minimal compact $\G$-space. Then every continuous $\G$-map $\{\phi_u\}_{u\in \Go}$ from $\Y$ into $\mathcal{P}(\X)$ has $\X$ as its range, i.e., for all $u\in \Go$, $\phi_u(Y_u)=X_u$. Equivalently, every $\G$-morphism from the groupoid dynamical system $ \{C(X_u)\}_{u\in \Go}$ into the groupoid dynamical system $\{C(Y_u)\}_{u\in \Go}$ is an $*$-isometric $\G$-morphism. Moreover, there is at last one such map.
\end{lem}
\begin{proof}
Let $\Phi=\{\phi_u\}_{u\in\Go}: \Y\To \mathcal{P}(\X)$ be a $\G$-map. The $\G$-subspace $\{\phi_u(Y_u)\}_{u\in\Go}$ of $\mathcal{P}(\X)$ contains $\X$. Since $\Y$ is minimal, the $\G$-subspace $\{\phi_u^{-1}(X_u)\}_{u\in\Go}$ coincides with $\Y$. Therefore, for any $u\in \Go$, $\phi_u(Y_u)=X_u$ and the $\G$-morphism $\Phi$ from
 $\{C(X_u)\}_{u\in\Go}$  into $\{C(Y_u)\}_{u\in\Go}$ is a $\G$-isometry. If there are two such maps $\Phi$ and $\Psi$, then $\{(\phi_u+\psi_u)/2\}$ is also a $\G$-map and hence has $\X$ as its range. Since  $\delta_x$ is an extreme point of $P(X_u)$, for any $u$, $\phi_u=\psi_u$ on $X_u$.
\end{proof}
\begin{defn}
The Furstenberg boundary  $\partial _F \G$  is a $\G$-boundary which is universal in the sense that it has
every $\G$-boundary   as a $\G$-quotient.
\end{defn}

Such a maximal $\G$-boundary exists: Take
the family $\{ \X_i\}_{i\in I} $  of all $\G$-boundaries (up to $\G$-isomorphism). By an argument similar to the one in the group case \cite{fur}, one can show that this forms a set, and we could consider Cartesian products. By Lemma \ref{lemm2},
$\prod_{i\in I}  \X_i$  is strongly proximal. Suppose that  $\partial _F \G$ is a
 minimal $\G$-subspace of $\prod_{i\in I}  \X_i$, which exists by Zorn lemma.  Since every $\G$-subspace of a strongly proximal $\G$-space is strongly proximal, $\partial_F \G$ is a $\G$-boundary and every $\G$-boundary is a quotient of $\partial_F \G$. Also, by Lemma \ref{lemm3},   such a maximal $\G$-boundary is unique.

Let $\X=\{X_u\}_{u\in \Go}$ be a $\G$-space and $w\in \Go$. For  $u\in \Go$, put  $S_u=\ell^{\infty}(\G^u_w)$ if $\G^u_w\neq \emptyset$, and  $S_u=C(X_u)$, otherwise.
Note that, for $g\in \G$, $\G_w^{d(g)}\neq \emptyset$ if and only if $\G_w^{r(g)}\neq \emptyset$. Define
$\alpha_g: C(X_{d(g)})\To C(X_{r(g)})$ by $\alpha_g(f)(x)=f(g^{-1}\cdot x)$, for $f\in C(X_{d(g)})$ and $x\in  X_{r(g)}$, and $\alpha_g: \ell^{\infty}(\G^{d(g)}_w)\To \ell^{\infty}(\G^{r(g)}_w)$ by $\alpha_g(f)(h)=f(g^{-1}h)$ for $f\in \ell^{\infty}(\G^{d(g)}_w)$ and $h\in \G^{r(g)}_w$. Then $(\G, \{S_u\}_{u\in\Go}, \alpha)$ is a groupoid dynamical system. For  $\mu\in P(X_{w})$,   define
$P_{\mu}^u: C(X_u)\To \ell^{\infty}(\G^u_w)$ by
\[P_{\mu}^u (f)(g)=\int_{X_w}f(g\cdot x)~d\mu(x),\]
if $\G^u_w\neq \emptyset$, and by $P_{\mu}^u=id_{S_u}$, otherwise. Then $\{P_w^u\}_{u\in\Go}$ is a $\G$-morphism. Take $g\in \G$ with $\G^{d(g)}_w\neq \emptyset$. For $h\in \G^{r(g)}_w$  and $f\in C(X_{r(g)})$,
\begin{align*}
(g\cdot P_{\mu}^{d(g)})(h)&=P_{\mu}^{d(g)}(g^{-1}h)=\int_{X_w} f((g^{-1}h)\cdot x)~d\mu(x)\\
&=\int_{X_w} (g\cdot f)(h\cdot x)~d\mu(x)\\
&=P_{\mu}^{r(g)}(g\cdot f)(h).
\end{align*}
If $\G^u_w=\emptyset$, $P_{\mu}^d(g)=id_{S_{d(g)}}$ and $P_{\mu}^{r(g)}=id_{S_{r(g)}}$. Also, $g\cdot P_{\mu}^{d(g)}(f)=P_{\mu}^{r(g)}(g\cdot f)$.
We call the $\G$-morphism $\p_\mu=\{P_{\mu}^u\}_{u\in\Go}$ the {\it Poisson $\G$-map} associated  to $\X$ and $\mu$ .

The next two results extend \cite[Lemma 3.6]{kk} and \cite[Proposition 3.4, 3.6]{kk}.

\begin{lem}\label{llm}
Let $\G$ be a groupoid, let $\partial_H\G=\{\partial_H^u\G\}_{u\in\Go}$ be the Hamna boundary of $\G$ and $w\in \Go$. Then for every $\mu\in P(X_w)$, the Poisson $\G$-map $\p_{\mu}$ associated  to $\partial _H\G$  is a $\G$-isometry.
\end{lem}
\begin{proof}
Since $\{C(\partial ^u_H\G)\}_{u\in \Go}$ is the injective envelope of the trivial system $\{\mathbb{C}\}_{u\in\Go}$ and $P_{\mu}^u: C(X_u)\To S_u$  is  a u.c.p. $\G$-map, the $\G$-essentiality of $ \{C(\partial ^u_H\G) \}_{u\in \Go}$ implies that $P_{\mu}^u$ is an isometry for each $u\in \Go$.
\end{proof}

\begin{prop}
The action of $\G$  on the Hamana boundary $\partial _H\G$  is minimal and strongly proximal.
\end{prop}
\begin{proof}
  Suppose $\Y=\{Y_u\}_{u\in \Go}$ is a $\G$-subspace of $\partial _H\G$ and suppose that $$i_u :C(\partial^u_H \G)\To C(Y_u)$$  is  the restriction map.  By the essentiality of $\{C(\partial_H^u \G)\}_{u\in \Go}$, $i_u$ is an isometry and hence $Y_u=\partial^u_H\G$, for all  $u\in\Go$.

 Given $u, w\in \Go$ with $\G^u_w\neq \emptyset$  and  $\mu\in P(\partial _H^w \G)$, we show that, for every $x\in \partial_H^u\G$, $\delta_x$ is in the weak$^*$-closed convex hull of $\G^u_w\cdot \mu$. Otherwise, there exists $f\in C_+(\partial^u_H\G)$ and $r>0$ such that,
for all $g\in G^u_w$,
\[P_{\mu}^u(f)(g)=\int_{X_w} f(g\cdot x)~d\mu(x)=\langle f, g\cdot\mu\rangle\leq f(x)-r\leq \|f\|-r,\]
hence $\|P_{\mu}^u (f)\| \leq \|f\|-r$. By Lemma \ref{llm}, $P_{\mu}^u$  is an isometry, which is a contradiction. Hence $\partial_H^u\G$ is contained  in the weak$^*$-closed  convex hull of $\G^u_w\cdot\mu$. But the weak$^*$-closed  convex hull of $\partial _H^u \G$ is $P(\partial_H^u\G)$, thus by the Krein-Milman theorem, $\partial _H^u \G\subseteq \overline{\G^u_w\cdot \mu}$.
\end{proof}

Now we are ready to prove the main result of this section, which extends \cite[Theorem 3.11]{kk}.
\begin{thm}\label{main}
For every groupoid $\G$, $\partial_F\G=\partial_H\G$.
\end{thm}
\begin{proof}
Since $\partial_H\G$ is a $\G$-boundary, by the universal property  of the Furstenberg boundary, there exists a surjective $\G$-map $\mathcal{Q}=\{q_u\}_{u\in \Go}:\partial_F\G\To \partial_H\G$. The injectivity of the  system  $\{C(\partial_H^u)\}_{u\in \Go}$, gives a $\G$-map $\Phi=\{\varphi_u\}_{u\in \Go}:\partial_H\G \To \partial _F\G$ such that $\mathcal{Q}\circ \Phi=id_{\partial_H\G}$. By Lemma \ref{lemm3}, the only $\G$-map from $\partial_F\G$ into itself is the identity map. Hence $\Phi\circ \mathcal{Q}=id_{\partial_F\G}$.
\end{proof}

\section{The reduced crossed product and its injective envelope}

 Let  $H$ be a Hilbert space and $X$ be a set. Suppose that $\ell^2(X,H)$  is  the set of all functions $\xi: X\To H$ with $\sum_{x\in X}\|\xi(x)\|^2< \infty$, then $\ell^2(X,H)$  with the pointwise operations and the inner product \[\langle  \xi, \zeta \rangle=\sum_{x\in X}\langle \xi(x), \zeta(x) \rangle\]
 is a Hilbert space. For $x\in X$ and $h\in H$ define
\[\delta_{x,h}(t)= \left\{
   \begin{array}{ll}
    h & t= x \\
    0 & t\neq x.
   \end{array}
 \right.\]
Then $\delta_{x,h}\in \ell^{2}(X,H)$ and $\|\delta_{x,h}\|= \|h\|$. Any $T\in B(\ell^{2}(X,H))$ induces a bounded  function
$\varphi=\phi_T :X\times X\To H$ by $\varphi(x,y)h=T(\delta_{y,h})(x)$. Clearly $\| \varphi\|\leq \| T\| $ and
\begin{equation}\label{e1}
(T\xi)(x)=\sum_{t\in X} \phi(x,t)(\xi(x)).
\end{equation}
Conversely, if $\varphi :X\times X\To H$  is  a bounded function, formula \eqref{e1} defines a bounded operator $T=T_{\varphi} $ on $B(\ell^{2}(X,H))$ with $\varphi_{T_{\varphi}}=\varphi$, that is,     $T_{\varphi_{T}}=T$.
Suppose that  $\mathcal{F}(X)$ is the family of all finite subsets of $X$. For any
$F\in \mathcal{F}(X)$ and $\varphi\in \ell^{\infty}(X\times X, B(H))$, the restriction $\varphi_F$ of $\varphi$ to $F\times F$ induces an operator on $H^{|F|}$. Moreover, such a $\varphi$ induces an operator on $B(H)$ if and only if
\begin{equation}\label{e2}
||\varphi||=\sup\left\{||\widetilde{\phi_F}||: F\in \mathcal{F}(X)\right\}<\infty.
\end{equation}

Net let us remind some basic facts about monotone completion of C*-algebras. Let $A$ be a C*-algebra then its self-adjoint part $ A_{sa}$ has a natural partial ordering.
If each norm bounded, increasing net in  $A_{sa}$  has a least upper bound then
A is said to be {\it monotone complete}. In this case, $A $ is
unital. Let $A$ and $B$ be  C*-algebras. A positive linear map $\phi :A\To B$ is called {\it normal} if for every norm bonded increasing net
$\{a_i\}_{i\in I}$  in $A_{sa}$ with $a=\sup_{i\in I}a_i$  we have $ \sup_{i\in I}\phi(a_i)=\phi(a)$.

Each C*-algebra $A$  has a unique regular monotone completion $ \overline{A}$ and injective envelope  $I(A)$ with $A\subseteq \overline{A}\subseteq I(A)$,   such that the inclusion maps $ A\hookrightarrow \overline{A}\hookrightarrow I(A)$ are normal.

Let $A$ be a monotone complete C*-algebra. For an increasing net $\{a_i\}_{i\in I}$  in $A_{sa}$ with $a=\sup_{i\in I}$, we write $a_i\nearrow a (\mathcal O)$ or $-a_i\searrow -a (\mathcal O)$. A net $\{a_i\}_{i\in I}$ in $A$ order-converges to $a$ , written $\mathcal O$-$\lim_i a_i=a$,  if there are bounded nets $\{a_i^{(k)}\}_{i\in I}$ , $\{b_i^{(k)}\}_{i\in I}$ in $A_{sa}$ and elements  $a^{(k)}$ in $A_{sa}$, $k=0,1, 2, 3$,  such that
$0\leq a_i^{(k)}-a^{(k)}\leq b_i^{(k)}\searrow 0 (\mathcal O)$ and $a_i=\sum_{k=0}^3 i^k a_i^{(k)},\  a=\sum_{k=0}^3 i^k a^{(k)}.$
Given a monotone complete C*-subalgebra  $A$ of $B(H)$ and von Neumann subalgebra $M$ of $B(K)$, the monotone tensor product $A\overline{\otimes }M$ of $A$ and $M$ is the monotone closure of $A\odot M$ in the Fubini product $\mathfrak{F}(A, M)$, that is, the smallest monotone closed C*-subalgebra containing $A\odot M$ in $\mathfrak{F}(A, M)$. The monotone complete tensor product $A\overline{\otimes }M$ does not depend on the underling Hilbert space $H$ and $K$. It is the  monotone closure of $A\odot M$. More generally, if $A$ is a monotone closed C*-subalgebra of a monotone complete C*-algebra $B$ and $M$ is a von Neumann subalgebra of a von Neumann algebra $N$, then $A\overline{\otimes} M $ is the monotone closure of $A\odot M$ in $B\overline{\otimes}N$.

Suppose that $A$ is a monotone complete C*-subalgebra of $B(H)$ and $X$ be any set. We consider the monotone tensor product
$A\overline{\otimes} B(\ell^2(X))$.  For our purposes here, it
is enough to observe that each element of the monotone tensor product $A\overline{\otimes} B(\ell^2(X))$
has a representation as a matrix over $A$, that is, each element of $A\overline{\otimes} B(\ell^2(X))$  is in the form
$\varphi: X\times X\To A\subseteq B(H)$ satisfying \eqref{e2}.
The involution and multiplication in $A\overline{\otimes} B(\ell^2(X))$ are defined as follows:
\[\varphi^*(x, y)=\varphi(x, y)^*,\hspace{1cm} \varphi\circ\psi(x, y)=\mathcal O{\text -}\sum_{t\in X} \varphi(x, t)\psi(t, y).\]

\begin{defn}
Let $\gad$ be a groupoid dynamical system. Define
\[\mathcal{A}^{(\alpha)}=\left\{ f\in \Gamma_{\infty}(\Go, \mathcal{A}):  f(r(g))=\alpha_g(f(s(g)))\right\}.\]
We call $\mathcal{A}^{(\alpha)}$ the {\it fixed point algebra} associated to the  system $\gad$.
\end{defn}

\begin{prop}\label{p1}
Let $\gad$ be a groupoid dynamical system such that for each $u\in \Go$, $A_u$ is a monotone complete C*-algebra. Then the fixed point algebra $\mathcal{A}^{(\alpha)}$ is a monotone complete C*-algebra.
\end{prop}
\begin{proof}
It is clear that $\mathcal{A}^{(\alpha)}$ is a C*-subalgebra of  $\Gamma_{\infty}(\Go, \mathcal{A})$. For the monotone completeness, suppose that $\{f_j\}_{j\in J}$  is  a norm bounded  increasing net in $\mathcal{A}^{(\alpha)}$. Then, for each $u\in \Go$ , $\{f_j(u)\}_{j\in J}$ is an increasing net in $A_u$,, and hence it has the least upper bound, say $f(u)\in A_u$. Then $f\in \Gamma_{\infty}(\Go, \mathcal{A})$ and for any $g\in \G$,
\begin{align*}
\alpha_g(f(s(g)))&=\alpha_g(\sup_j f_j(s(g)))\\
&=\sup_j \alpha_g(f_j(s(g)))\\
&=\sup_j f_j(r(g))=f(r(g)).
\end{align*}
\end{proof}
\begin{prop}\label{prop1}
Let $\mathcal{A}=\gad$ be an  injective groupoid dynamical system. Then $\mathcal{A}^{(\alpha)}$   is an injective C*-algebra.
\end{prop}
\begin{proof} By assumption, there exists a $\G$ -morphism $\phiu$  from the dynamical system $( \G, \{\Gamma_{\infty}(\G^u,s^* \mathcal{A})\}_{u\in \Go},\ell )$ into $\mathcal{A}$  such that  $\varphi_u\circ \im_u=\id _{A_u}$, for any $u\in\Go$.  Suppose that $\widetilde{ \A}=\{\Gamma_{\infty}(\G^u,s^*  \mathcal{A})\}_{u\in \Go}$  and define
\[\tau:  \Gamma_{\infty}(\Go, \mathcal{A})\To \Gamma_{\infty}(\Go,  \widetilde{\A} ); \ \tau(f)(u)(x)=f(s(x)),\]
and
\[\Psi: \Gamma_{\infty}(\Go, \mathcal{A})\To \Gamma_{\infty}(\Go, \mathcal{A}); \
\Psi(f)(u)=\varphi_u(\tau(f)(u)).\]
For $g\in \G, x\in \G^{r(g)}$ and $f\in \Gamma_{\infty}(\Go, \mathcal{A})$, we have
$$\ell _g(\tau(f)(u))(x)=\tau(f)(u)(g^{-1}x))=f(s(g^{-1}x))=f(s(x))=
\tau(f)(r(g))(x),$$ thus $\ell _g(\tau(f)(u))=\tau(f)(r(g))$. Therefore,
\begin{align*}
\alpha_g(\Psi(f)(s(g)))&=\alpha_g(\varphi_{s(g)}(\tau(f)(s(g))))\\
&=\varphi_{r(g)}(\ell _g(\tau(f)(s(g))))\\
&=\varphi_{r(g)}(\tau(f)(r(g)))\\
&=\Psi(f)(r(g)),
\end{align*}
that is, $\Psi(f)\in \mathcal{A}^{(\alpha)}$. For $f\in \mathcal{A}^{(\alpha)}$,
$\tau(f)(u)(x)=f(s(x))=\alpha_{x}^{-1}(f(u))=\im_u(f)(x)$, for each $x\in \G^u$, thus $\tau(f)(u)=
\im_u(f(u)).$ Hence
$$\Psi(f)(u)=\varphi_u(\tau(f)(u))=\varphi_u(\im_u(f(u))=f(u).$$ Therefore, $\Psi $ is a conditional expectation form $\Gamma_{\infty}(\Go, \mathcal{A})$  onto $\mathcal{A}^{(\alpha)}$.
Since
\[\Gamma_{\infty}(\Go, \mathcal{A})=\bigoplus_{u\in \Go}^{\ell^{\infty}}A_u\] and
each $A_u$ is  an injective C*-algebra,   $\mathcal{A}^{(\alpha)} $ is an  injective C*-algebra.
\end{proof}

Let $\gad$ be a groupoid dynamical system such  that for  any $u$, $A_u$ is monotone complete. For $u\in \Go$, set $\mathcal{A}^{(u)}=A_u\overline{\otimes }B(\ell^2(\G_u))$. So each element of  $\mathcal{A}^{(u)}$ is represented by a function $\varphi=\varphi_u: \G_u\times \G_u\To A_u$ such that \eqref{e2} is satisfied. For $g\in G$, define $\widetilde{\alpha}_g: \mathcal{A}^{(s(g))}\To \mathcal{A}^{(r(g))}$ by
\[\widetilde{\alpha}_g(\varphi)(x, y)=\alpha_g(\varphi(x.g, y.g)).\]
It is not hard to see that
\[\A^{\otimes}=(\G, \{\mathcal{A}^{(u)}\}_{u\in \Go}, \widetilde{\alpha})\]
is a groupoid dynamical system. The fixed point subalgebra associated to this goupoid dynamical system is called the monotone crossed product of $\mathcal{A}$ by $\G$ and is denoted by $M(\G, \mathcal{A})$. By Proposition \eqref{p1}, $M(\G, \mathcal{A})$ is  monotone complete. Next we show that if $\A$ is an injective groupoid dynamical system, then $M(\G, \mathcal{A})$  is an injective C*-algebra.
Let $u\in \Go$ and
\[\widetilde{\im}_u : A^{(u)}\To \Gamma_{\infty}(\G^u, s^* \A^{\otimes})\] be the embedding defined as above, that is,
\[\widetilde{\im}_u(F)(x)=\widetilde{\alpha}_x^{-1}(F).\]
\begin{thm}
Let $\A=\gad$ be an injective groupoid dynamical system. Then $M(\G, \A)$ is an injective C*-algebra.
\end{thm}
\begin{proof}
Since $M(\G, \A)$ is the fixed point algebra associated to the groupoid dynamical system $\A^{\otimes}$,  it is sufficient to prove that  $\A^{\otimes}$is injective. Since $\A$ is $\G$-injective,  there exists a $\G$-map $\phiu$ from
$ (\G, \{ \Gamma_{\infty}(\G^u, s^*\A)\}_{u\in \Go}, \ell)$ into $\A$ such that for any $u\in \Go$, $\phi_u\circ\im_u=\id_{A_u}$.
Define the map
\[\id_u\otimes 1: A_u\To A^{(u)},\]
by
\[\id_u\otimes 1(a)(x,y)=\left\{
\begin{array}{lr}
a&x=y\\
0& x\neq y.
\end{array}
\right.
\]
Since $\A$ is $\G$-injective, there exists a $\G$-map $\theu$  from $\A^{\otimes}$ into $\A$ such that, for each $u\in \Go$,
$\theta_u\circ\id_u\otimes 1=\id_{A_u}$.
Also, define
\[\Theta_u: \Gamma_{\infty}( \G^u, s^* \A^{\otimes})\To \Gamma_{\infty}( \G^u, s^* \A),\]
by
$\Theta_u (F)(g)=\theta_{s(g)}(F(g))$.
Then, for  $u\in \Go$, $\theta_u\circ\widetilde{\im}_u\circ \id_u\otimes 1=\im_u$. Set $\widetilde{\varphi}_u =\id_u\otimes 1\circ \varphi _u\circ \theta_u$. Then $\widetilde{\Phi}=\{\widetilde{\varphi}_u\}_{u\in\Go}$ is a $\G$-map from the dynamical system $(\G, \{\Gamma_{\infty}(\G^u, s^* \A^{\otimes})\}_{u\in\Go}, \ell)$ into $\A^{\otimes}$ such that $\widetilde{\varphi}_u\circ\widetilde{\im}_u=\id_{A^{(u)}}$. Thus  $\A^{(\alpha)}$ is $\G$-injective.
\end{proof}
 Let $s: X \To X_0$ be a surjective map and $\{H_u\}_{u\in X_0}$ be a family of Hilbert spaces. Let  $\ell^2(X, s^*\mathcal{H})$  be the set of all  functions $\xi: X\To \cup_{u\in X_0} H_u$ such that $\xi(x)\in H_{s(x)}$ and
\[\|f\|=\sum_{x\in X} ||f(x)||^2<\infty.\]
Then $\ell^2(X, s^*\mathcal{H})$ is a Hilbert space with the pointwise operations and the inner product
\[\langle \xi, \eta\rangle=\sum_{x\in X} \langle \xi(x), \eta(x)\rangle,\]
which  is
canonically isomorphic to the $\ell^2$-direct sum
\[\bigoplus_{u\in X_0}\ell^2(s^{-1}(u), H_u).\]

Suppose that  $\A=\gad$  is a groupoid dynamical system  and  suppose that for any $u\in \Go$, $H_u$  is  a Hilbert space such that $A_u\subset B(H_u)$. For $f\in \Gamma_c(\G, r^*\mathcal{A})$ define $\Pi(f):\ell^2(\G , s^*\mathcal{H})\To \ell^2(\G, s^*\mathcal{H})$  by
\begin{equation}
(\Pi(f)\xi)(x)=\sum_{t\in \G^{r(x)}} \alpha_x^{-1}(f(t))\xi(t^{-1}x).
\end{equation}
This is a faithful representation for $\Gamma_c(\G, r^*\mathcal{A})$. The norm closure $\mathcal{A}\rtimes_r \G$ of
$\Pi(\Gamma_c(\G, r^*\mathcal{A}))$ in  $B(\ell^2(\G, s^*\mathcal{H}))$ is called the reduced crossed product of $\mathcal{A}$ by $\G$.

Let $\A=\gad$ be a groupoid dynamical system such that for any $ u\in \Go$, $A_u$ is monotone complete. For $u,v\in \Go$ and $a\in A_v$,  define
\[\pi_u^v(a) : \G_u\times \G_u\To A^u\]
by
\begin{align*}
\pi_u^v(a)(x, y)=
\left\{
                \begin{array}{ll}
                \alpha_x^{-1}(a), & x=y \in \G^v_u \\
                  0, & \hbox{otherwise,}
                \end{array}
              \right.
\end{align*}
Also, for $g\in G$, define
\begin{align*}
 &\lambda_u(g): \G_u\times \G_u\To A_u \\
& \lambda_u(g)(x, y)=\left\{
                   \begin{array}{ll}
                     1, & xy^{-1}=g \\
                    0, & \hbox{otherwise.}
                   \end{array}
                 \right.
\end{align*}
Then $\pi^v=(\pi_u^v)_{u\in \Go}$ and $ \lambda(g)=(\lambda_u(g))_{u\in \Go}$ are in $M(\G, \A)$. For
$a\in A_v$ and $\xi=(\xi_u)_{u\in \Go}\in \ell^2(\G, s^*\mathcal{H})$,   we have
$\pi^v(a)\xi=(\pi_u^v (a)\xi_u)_{u\in \Go} $  and $\pi_u^v(a)\xi_u(x)=\alpha_x^{-1}(a)\xi_u(x)$. Also, for $g\in \G$, $\lambda_g\xi=( \lambda_u(g)\xi_u)_{u\in \Go}$ and
\[ \lambda_u(g)\xi_u(x)=\left\{
\begin{array}{lr}
\xi_u(g^{-1}x)& r(x)=r(g)\\
0,& \hbox{otherwise.}
\end{array}
\right.
\]
Hence,
\[(\pi_u^v(a)\circ \lambda_u(g))\xi_u(x) =\left\{
                           \begin{array}{lr}
                            \alpha_x^{-1}(a)\xi_u(g^{-1}x) , & r(x)=r(g)=v   \\
                             0, & \hbox{otherwise.}
                           \end{array}
                         \right.\]
Since the reduced crossed product $\mathcal{A}\rtimes_{r} \G$ is generated by the set
\[\left\{\pi^v(a)\lambda(g): v\in \Go a \in \A _v, g\in \G, \right\},\]
  $M(\A, \G)$ contains $\mathcal{A}\rtimes_{r} \G$ .

For each $v\in \Go$, $\pi^v(A_v)$ is a C*-subalgebra of $M(\G, \A)$ and for $g\in \G $ and $a\in A_{s(g)}$ we have
\[\lambda(g) \pi^{s(g)}(a) \lambda^*(g)=\pi^{r(g)}(\alpha_g(a)).\]
Let $\alpha'_g(\pi^{s(g)}(a))=
\pi^{r(g)}(\alpha_g(a))$, then $(\G, \{\pi^v(A_v)\}_{v\in \Go}, \alpha')$ is a groupoid dynamical system,  isomorphic to the original system $\A$.

Now we are ready to prove the main result of this section, which is essential in further development of the theory \cite{ab}.

\begin{thm}
Suppose that $\A=\gad$ be a groupoid dynamical system. Then
\[I_{\G}(\A)\rtimes_r \G\subseteq I(\A\rtimes_r \G ).\]
\end{thm}
\begin{proof}
Since $\A\rtimes_r \G\subseteq I_{\G}(\A)\rtimes_r \G\subseteq M(\G, I_{\G}(\A))$ and
$M(\G, I_{\G}(\A))$ is an injective C*-algebra,
we may consider $ I(\A\rtimes_r \G )\subseteq M(\G, I_{\G}(\A))$. The inclusion map $$j: \A\rtimes_r \G\To I_{\G}(\A)\rtimes_r \G$$ extends to a completely positive map $$\psi: I_{\G}(\A)\rtimes_r \G\To  I(\A\rtimes_r \G )\subseteq M(\G, I_{\G}(\A)).$$  Since $\psi $ is completely positive and preserves $\lambda(G)$,  by \cite[3.1]{ch}, for any $ f\in  I_{\G}(\A)\rtimes _r \G$, $\psi(\lambda(g)f\lambda(g)^*)= \lambda(g)\psi(f)\lambda(g)^*$.
 For $v\in \Go$, set $\psi_v=\psi |_{\pi^v(I_{\G}(\A)_v)}$  and suppose that
 $ \rho_v: M(\G, I_{\G}(\A))\To \pi^v(I_{\G}(\A)_v)\subseteq M(\G, I_{\G}(\A))$  is the map defined by
\[\rho_v((F_u)_{u\in \Go})=\pi^v(F_v(v, v)).\]

Thus $\rho_v$ is a conditional expectation and if $\rho_v((F_u)_{u\in \Go}^*(F_u)_{u\in \Go})=0$ then for any $u\in \Go$ and $x\in \G_u^v,y\in\G_u$, $F_u(x,y)=0$.

 For $g\in \G$  and $F\in M(\G, I_{\G}(\A))$,
\begin{align*}
\tilde{\rho}_{r(g)}(\lambda(g)(F_u)_{u\in \Go}\lambda(g)^*)
&= \pi^{r(g)}(\lambda_{r(g)}(g)\circ F_{r(g)}\lambda_{r(g)}(g)^*)\\
&= \pi^{r(g)}(F_{r(g)}(g^{-1}, g^{-1}))\\
&= \pi^{r(g)}(\alpha_g(F_{s(g)}(s(g), s(g)))\\
&=\lambda(g)\pi^{s(g)}(F_{s(g)}(s(g), s(g)))\lambda(g)^*\\
&=\lambda(g) \tilde{\rho}_{s(g)} \lambda(g)^*.
\end{align*}
Since $(\G, \{\pi^v(I_{\G}(A)_v)\}_{v\in \Go}, \alpha') $ is a $\G$-essential extension of the dynamical system $(\G, \{\pi^v(A_v)\}_{v\in \Go}, \alpha')$ and $\rho_v$ and $\psi_v$ preserve $\pi^v(A_v)$, $\widetilde{\rho}_v\circ \psi_v(x)=x$, for each $x\in \pi^v(I_{\G}(A)_v)$.  By a similar arguments as in the proof of  \cite[Lemma 3.3]{ham1}, we conclude that
\[ \rho_v((\psi(\pi^v(a))-\pi^v(a))^*(\psi(\pi^v(a))-\pi^v(a)))=0,\]
thus, for any $u\in \Go$ and $x\in \G^v_u, y\in \G_u, $
\[\psi(\pi^v(a))_u(x,y)=\pi_u^v(a)(x,y),\]
and hence
$\psi(\lambda_v\pi^v(a))=\lambda_v\psi(\pi^v(a))=\pi^v(a).$  This means that $\pi^v(a)$ is in the range of $\psi$ . Since $I_{\G}(\A)\rtimes_r \G$ is generated by the operators $\lambda_g$ and $\pi^v(a)$, we get $I(\A\rtimes \G )\supseteq I_{\G}(A)\rtimes_r \G$.

\end{proof}

\end{document}